\documentclass[11pt]{article}
\usepackage{amssymb}
\usepackage{amsmath}
\usepackage{amsthm}
\usepackage{epsfig}
\usepackage[mathscr]{eucal}
\def\RR{{\mathbb R}}

\def\beqns{\begin{eqnarray*}}
\def\eeqns{\end{eqnarray*}}
\def\beqn{\begin{eqnarray}}
\def\eeqn{\end{eqnarray}}

\def\no{\noindent}
\def\no{\noindent}

\def\ca{{\mathcal A}}

\def\dis{\displaystyle}
\def\no{\noindent}

\usepackage{amssymb,amsmath,amsthm}

\newcommand{\R}{\mathbb R}
\newcommand{\N}{\mathbb N}

\newtheorem{theorem}{Theorem}[section]
\newtheorem{lemma}[theorem]{Lemma}

\newtheorem{corollary}[theorem]{Corollary}
\newtheorem{proposition}[theorem]{Proposition}

\renewenvironment{proof}{{\noindent \bf
Proof:}}{\hfill\qed\bigskip}

\numberwithin{equation}{section}

\numberwithin{equation}{section}

\numberwithin{equation}{section}

\begin{document}

\title{Optimal order Jackson type inequality for scaled Shepard approximation }
\author{Steven\ Senger,  Xingping\ Sun, and  Zongmin\ Wu}

\date{}

\maketitle
\begin{abstract}
We study a variation of the Shepard \cite{shepard} approximation scheme by introducing a dilation factor into the base function, which synchronizes with the Hausdorff distance between the data set and the domain. The novelty enables us to establish an optimal order Jackson \cite{jackson} type error estimate (with an explicit constant) for bounded continuous functions on any given convex domain. We also improve en route an upper bound estimate due to Narcowich and Ward \cite{NW} for the numbers of well-separated points in thin annuli,  which is of independent interest.
\end{abstract}

\no{\bf Key Words and Phrases:} Hausdorff distance, Jackson type error estimate, quasi-interpolation operator, quasi-uniformity, rational formations, well-separateness.

\bigskip
\no{\bf AMS Subject Classification 2010:} 41A17, 41A35, 41A36, 41A46.

\newcommand{\cx}{\mathcal X}
\newcommand{\ck}{\mathcal K}
\newcommand{\ct}{\mathcal T}
\medskip

\section{Introduction}
When dealing with real world problems with high degrees of complexity, scientists often observe that unknown target functions are imprecise and elusive, and that data acquired on them do not always reflect their true nature. This can be caused by a host of known and unknown reasons. To name just a few, one may encounter reading and interpreting errors, lost in translation, instrument failures or malfunctions. In the literature, this is often referred to as the ``noisy data" phenomenon (\cite{behe}). Under these circumstances, models established by computationally-expensive algorithms often do not  survive the test of cross validation and sometimes  fail outright to reflect reality in ways they are designed for. Long and hard work devoted to the establishment of the models is quickly rendered worthless. To make meaningful decisions, one needs to use
the best available data to build multiple models and test them against newly acquired data. The process is often repeated numerous times,  as the so called ``best available data" and model selection criteria may also subject to uncertainty of various degrees.
%If one has finally found a winner, do not rush to open the champaign bottle yet. Be humble. In this kind of computing arena, any success can be temporary and fleeting.

The challenging computing environment has ruled out the employability of most of the interpolation methods and many of the quasi-interpolation methods. An interpolation procedure entails solving a large linear system, which is often expensive and slow. In addition, interpolating ``noisy data" is like playing a meaningless  hide-and-seek game during which one seldom knows the boundary between ``over-fitting" and ``under-fitting". Many quasi-interpolation schemes rely on elaborated mathematical procedures to find the values of the parameters, which is unpractical in reality.

Shepard \cite{shepard} proposed in $\R^2$ the following approximation scheme. Suppose function values $f(x_j)\;
(1 \le j \le n)$ are available at the scattered sites $\{x_1, \ldots, x_n\}$ in a domain $X$. Then a function of the form:
\[
x \mapsto S_{\Phi,n}(x):=[\Phi_n(x)]^{-1} \sum^n_{j=1}f(x_j)\ \phi(|x - x_j|),
\]
is constructed to approximate the target function $f$ in $X$. Here $\phi: x \mapsto |x|^{-\lambda}, \; \lambda \ge 1$, and $\Phi_n(x)=\sum^n_{j=1}\ \phi(x - x_j)$, in which $|x|$ denotes the Euclidean norm of $x \in \R^2$. This procedure has since become known as ``Shepard approximation", and variations of it have been studied in \cite{bede}, \cite{chen-cao}, \cite{cost}, \cite{szab},  \cite{tikk}, \cite{wendland}, \cite{wu-1}, \cite{wu}, and the references therein.

In a nutshell, a Shepard approximation scheme employs rational formations of shifts of an appropriately-selected base function to approximate a target function, and its efficiency epitomizes in the reproduction of constants. Because of the singularity of the base function at zero, the original Shepard approximant interpolates  values of the target function at $x_j,\; 1 \le j \le n$. That is,
$S_{\Phi,n}(x_j)=f(x_j),\; 1 \le j \le n$. However, for most other choices of base functions, the interpolation feature is lost, and the resulted approximants are called ``quasi-interpolants" in the literature; see \cite{chen-cao}, \cite{wu-1}, \cite{wu}. Except for the cases in which the base function is compactly supported (see \cite{wendland}, \cite{wu-1}, \cite{wu}), deriving optimal order error estimates for a Shepard approximation scheme has been uncommon. In this paper, we study a new variation of Shepard approximation by introducing a dilation factor into the base function, which synchronizes with the Hausdorff distance between the data set and the domain. The novelty enables us to establish an optimal order Jackson \cite{jackson} type error estimate (with an explicit constant) for bounded continuous functions.
 The proof requires decomposing the domain as the union of concentric thin annuli with no common interior.
 This is a standard technique in many branches of analysis. Notably, the technique has recently been applied by authors of \cite{HNW} and \cite{HNSW} in bounding the $L_\infty$-norms of interpolation operators and the least square operators associated with radial basis functions.
The following question arises naturally: how many well-separated points can be put inside a specified annulus?
For the special case in which the annuli have thickness $\delta$ and inner radius $j\delta, \; j \in \N$, where $\delta$ is the separation radius of the data set, Narcowich and Ward \cite{NW} gave the upper bound estimate: $3^d\ j^{d-1}$, where $d$ is the dimension of the ambient space. In our situation, however, the annuli have thickness precisely the Hausdorff distance between the data set and the domain, which is usually larger than the separation radius of the data set.
 Adapting their method in this setting has resulted in an unexpected large constant that has a fast-growing exponential rate with respect to the dimension. In order to refine this, we employ the Gauss hypergeometric functions, and engage in a more labor-intensive procedure to obtain a tighter estimate (Proposition \ref{lemma}). We also derive a new upper bound estimate for the above special case (Corollary \ref{NW}). Besides their utility in this paper to obtain optimal order Jackson type error estimate, these upper bound estimates provide information on the minimal densities of packing spheres or spherical wedges in annuli of $d-$dimensional space. Therefore, their asymptotic behaviors with respect to both $j$ and $d$ need to be closely watched.  The results of Proposition \ref{lemma} and Corollary \ref{NW} show that in packing spheres or spherical wedges in annuli, both the size of the spheres (or spherical wedges) and that of the annuli are sensitive matters. This is in contrast to density results from conventional sphere packings in Euclidean spaces; see \cite{conway} and \cite{cohn}.
 Details of the above discussion will be given in
 Section 2.  The main result of the paper (Jackson type error estimate) and its proof will be given in Section 3.
  A focus of the current paper is to provide an readily-implementable and yet efficient approximation method. As such, we write the paper in a style catering to the needs of practitioners in the field. Among other efforts devoted to obtain the results, we derive all the constants explicitly in closed forms.

\section{Numbers of well-separated points in thin annuli}

To be savvy with using notations, we adopt the notations: (i) $|t|$ for the absolute value of a real number $t$; (ii) $|E|$ for the cardinality of a finite set $E$; (iii) $|x-y|$ the Euclidean distance between $x, y \in \R^d$;  and (iv) $|A|$ for the Lebesgue measure of a bounded Lebesgue-measurable set $A$. Readers can easily tell each individual meaning from the mathematical context therein.
Let  $X \subset \R^d$ be a convex domain. Let $\cx_n \subset X$ be a sequence of discrete point sets. If there is a constant $c$ independent of $n$, such that
\begin{equation}\label{well-separated}
 \inf_{\substack{x \ne y\\x,y \in \cx_n}} |x - y| \ge c n^{-1/d},
\end{equation}
then we say that the points of $\cx_n$ are well-separated, or equivalently that the point set $\cx_n$ is well-separated. We will use $q_{\cx_n}$ to denote half of the above infimum.
Let $h_{\cx_n}$ be the Hausdorff distance (\cite{munkres}) between the two sets $\cx_n$ and $X$. That is
\begin{equation}\label{hausforff-d}
  h_{\cx_n}:= \sup_{x \in X} \inf_{y \in \cx_n} |x - y|.
\end{equation}
If $\cx_n$ is well-separated, and if in addition there is a constant, $C >0$, independent of $n$, such that
\begin{equation}\label{fill-dist}
h_{\cx_n} \le C n^{-1/d},
\end{equation}
 then we say that $\cx_n$ is quasi-uniformly distributed in $X$. In the literature, $h_{\cx_n}$ is often called the ``fill-distance" of the set $\cx_n$ in the set $X$, and $q_{\cx_n}$ the separation radius of the point set $\cx_n$. It is obvious that $h_{\cx_n} \ge q_{\cx_n}$. Thus, the two constants $c, C$ as in Inequalities \eqref{well-separated} and \eqref{fill-dist} satisfy $C \ge c/2$. In the current paper,  we will reserve the two constants $c, C $ exclusively for the purposes as dictated in \eqref{well-separated} and \eqref{fill-dist}.
%
%   Let $\cx_n \subset X$ be a sequence of discrete point sets satisfying Inequalities \eqref{well-separated} and \eqref{fill-dist}.
An interesting example for such a set $\cx_n$ is the  hexagonal lattice in $ \R^2$ with the generating matrix
\[
n^{-1/2} \left(
               \begin{array}{cc}
                 \sqrt{3} & 0 \\
                 -1 & 2 \\
               \end{array}
             \right).
             \]
In this example, we have $q_{\cx_n}= n^{-1/2}$, and $ h_{\cx_n}= n^{-1/2} \frac{2}{\sqrt{3}}$.
If $X$ is bounded, then a quasi-uniformly distributed set $\cx_n \subset X $ is necessarily finite. In this case, the quasi-uniformity can be equivalently defined as follows. There exists a constant $\rho (d)$, depending only on  $d$, such that
 \[
 h_{\cx_n} / q_{\cx_n} \le \rho(d).
 \]
The above inequality is equivalent to that there are two constants $0<C_1(d) \le C_2(d)$, depending only on $d$, such that
\[
C_1(d)\ n^{-1/d} \le q_{\cx_n} \le h_{\cx_n}  \le C_2(d)  n^{-1/d}.
\]

We are interested in finding an upper bound for the number $|\cx_n\cap \ca| $, where $\ca $ is an annulus of outer radius $R$ and thickness $Cn^{-\frac{1}{d}}$.
Our basic approach in answering the above question is to use the pigeonhole principle. Let $\epsilon = c n^{-\frac{1}{d}}$.  Select any two points, $p$ and $q$ from $E_n$. Let $B(p, \epsilon)$ denotes the ball centered at $p$ and having radius $\epsilon$. Notice that, as points from $\cx_n$ are $\epsilon$-separated, the intersection
$B(p, \epsilon/2) \cap B(q, \epsilon/2)$ has empty interior. Therefore, it is tempting to get an upper bound on the number of points in the annulus, $\ca$, by packing spheres \cite{conway} of radius $\epsilon/2$ in $\ca$. We would proceed by using the pigeonhole principle to get:
\begin{equation}\label{badPHole}
\left| \cx_n\cap \ca \right| \leq \frac{|\ca|}{|B^d(\epsilon/2)|},
\end{equation}
where $B^d(\epsilon/2)$ is a ball in $\R^d$ of radius $\epsilon/2$.
However, independent of $c$, it is possible for points of $\cx_n$ to lie on or close to the boundary of $\ca$. These points do not enjoy a full measurement of volumes of their associated balls contributing to the calculation above based on the pigeonhole principle. This is illustrated in Figure \ref{ideal}.

To correct the estimate in \eqref{badPHole}, we need to estimate the minimal possible volume of the intersection (with the annulus) of a small ball whose center lies in the annulus.  The correct version of \eqref{badPHole} is therefore the following:
\begin{equation}\label{PHole}
\left| E_n\cap \ca \right| \leq \frac{|\ca|}{|B_\ca^d(\epsilon/2)|},
\end{equation}
where $B_\ca^d(\epsilon/2)$ is the minimal intersection of any ball of radius $\epsilon/2$ centered at a point in the annulus, or
\begin{equation}\label{inter}
|B_\ca^d(\epsilon/2)| = \min_{q\in \ca}|\ca \cap (B^d(\epsilon/2)+q)|.
\end{equation}\\
We are therefore packing spherical wedges of two different sizes in a specified annulus. We illustrate the situation for the case $d=2$ in Figure 2.
A simple convexity argument shows that the above minimum is attained by a ball $B^d(\epsilon/2)$ centered at a point on the outer boundary of $\ca$. It stands to reason that finding a tight lower bound for the quantity $|B_\ca^d(\epsilon/2)|$ has become the center piece of the puzzle.
\begin{figure}
\centering
\includegraphics[scale=1]{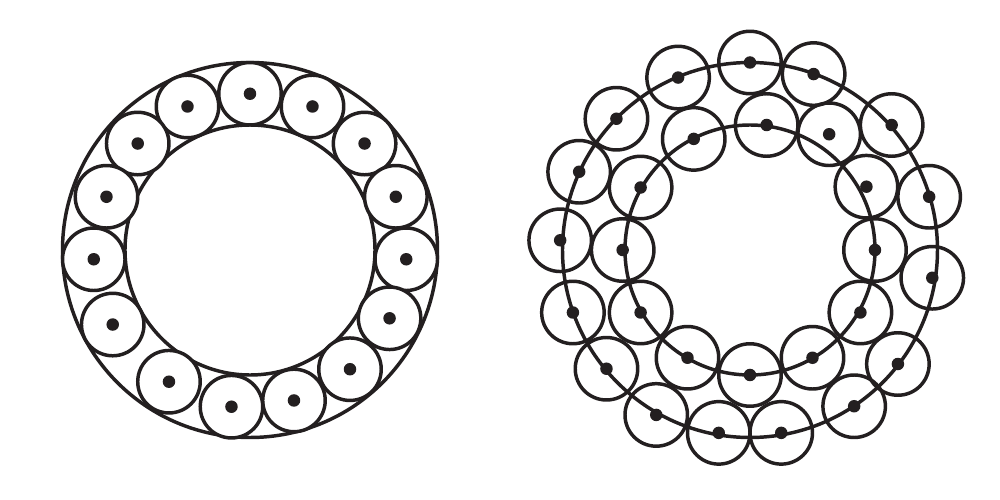}
\caption{On the left, we see a picture of the set consisting of 15 points in an annulus, each centered at mutually disjoint balls that are contained in the annulus. On the right, we see that more points centered at mutually disjoint balls can be in the annulus if we do not require that each ball is contained in the annulus.}
\label{ideal}
\end{figure}
To this goal, we will utilize the Gauss hyper-geometric function $\,_2 F_1 (a,b;c;z)$ defined by
\[
\,_2 F_1 (a,b;c;z)=\sum^\infty_{n=0}\frac{(a)_n (b)_n}{n!(c)_n} z^n, \quad |z| < 1,
\]
where $(a)_n$ is the (rising) Pochhammer symbol, which is defined by:
\[
(a)_n = \begin{cases} 1 & n = 0 \\ a(a+1) \cdots (a+n-1) & n > 0. \end{cases}
\]
We will need Euler's integral representation formula for hyper-geometric function; see \cite{askey}:
\begin{align} \label{euler}
\,_2F_1(a,b;c;z) = & \frac{\Gamma(c)}{\Gamma(b) \Gamma(c-b)} \int_0^1 x^{b-1} (1-x)^{c-b-1}(1-zx)^{-a}\ dx, \nonumber\\
 &\mbox{Re}(c) > \mbox{Re}(b) > 0, \quad |z| < 1.
\end{align}
While there is a wealth of literature devoted to hyper-geometric functions, we refer readers to \cite{askey} for a proof of the above Euler's integral representation formula for hyper-geometric function.

\begin{proposition}\label{dProp}
The volume of intersection of a $d$-dimensional ball of radius $R$ and a smaller $d$-dimensional ball of radius $r$ centered at a point in the larger ball is at least
\begin{align*}
 & \left(\sin^{d+1}\frac{A}{2}\right) \,
 \frac{2^{d+1} \,  \pi^\frac{d-1}{2} \, r^d}{(d+1) \, \Gamma\left(\frac{d+1}{2}\right)} \,_2 F_1 \left(-\frac{d-1}{2},\frac{d+1}{2}; \frac{d+3}{2}; \sin^2\frac{A}{2}\right),
\end{align*}
where $A=A(r,R)=\cos^{-1} \frac{r}{2R}.$
\end{proposition}

\begin{proof}
First, recall the following formula for the volume of a $d$-dimensional ball of radius $t$, which we denote $B^d(t)$, is:
\begin{equation}\label{ballVol}
\left|B^d(t)\right|: = V_d(t)=\frac{\pi^\frac{d}{2}}{\Gamma\left(\frac{d}{2}+1\right)}t^d
\end{equation}
We center the small ball (with radius $r$) at the origin, and the big ball (with radius $R$) at the point $(R, 0, \cdots,0)$, which gives rise to a situation where minimal intersection occurs. We denote the intersection of the two balls by $B_R^d(r)$, and write down:
\begin{equation}\label{dInt}
\left|B_R^d(r)\right|=\int_0^{\tilde{x}}V_{d-1}\left(\sqrt{2Rx-x^2}\right)dx+\int_{\tilde{x}}^rV_{d-1}\left(\sqrt{r^2-x^2}\right)dx = I_d(r,R) + II_d(r,R),
\end{equation}
in which $\tilde{x}=\frac{r^2}{2R}$. We have illustrated the cases $d=2,3$, respectively, in Figures 2 and 3.
To calculate $II_d(r,R),$ we write
\begin{align} \label{con-1}
II_d(r,R)
&=\frac{\pi^\frac{d-1}{2}}{\Gamma\left(\frac{d+1}{2}\right)}\int_{\tilde{x}}^r \left(r^2-x^2\right)^\frac{d-1}{2}dx
= \frac{r^d\, \pi^\frac{d-1}{2}}{\Gamma\left(\frac{d+1}{2}\right)}\int^1_{\frac{r}{2R}} \left(1-x^2\right)^\frac{d-1}{2}dx .
\end{align}
We now use the Gauss hyper-geometric function to evaluate the integral on the right hand side of the above equation. First we write
\begin{align*}
\int^1_{\frac{r}{2R}} \left(1-x^2\right)^\frac{d-1}{2}dx & =\int^A_{0}  \sin^d t dt, \quad \mbox{where} \quad A = \cos^{-1}\frac{r}{2R}.
\end{align*}
We then use the substitution
\[
\sin \frac{t}{2} = \sin \frac{A}{2} \sin \theta, \quad \frac12 \cos \frac{t}{2} dt = \sin \frac{A}{2} \cos \theta d \theta
\]
to reduce the integral.  We have
\begin{align} \label{con-2}
\int^A_{0}  \sin^d t dt & = 2^{d+1} \sin^{d+1}\frac{A}{2} \int^{\frac{\pi}{2}}_{0}\cos \theta \sin^d \theta \left( 1- \sin^2\frac{A}{2} \sin^2 \theta \right)^{\frac{d-1}{2}} d \theta \nonumber\\
& = 2^{d} \sin^{d+1}\frac{A}{2} \int^1_0 t^{\frac{d-1}{2}} \left( 1- t \sin^2\frac{A}{2} \right)^{\frac{d-1}{2}} dt.
\end{align}
We use Euler's integral representation (Eq. \eqref{euler}) for the hypergeometric function to write
\begin{align} \label{con-3}
\int^1_0 t^{\frac{d-1}{2}} \left( 1- t \sin^2\frac{A}{2} \right)^{\frac{d-1}{2}} dt = &
%\frac{\Gamma(\frac{d+1}{2})}{\Gamma(\frac{d+3}{2})}
\frac{2}{d+1} \,_2 F_1 \left(-\frac{d-1}{2},\frac{d+1}{2}; \frac{d+3}{2}; \sin^2\frac{A}{2}\right).
\end{align}
Combining Equations \eqref{con-1}, \eqref{con-2}, and \eqref{con-3}, we have
\begin{align} \label{combined}
II_d(r,R) &= \sin^{d+1}\frac{A}{2} \,
 \frac{2^{d+1} \, r^d\, \pi^\frac{d-1}{2}}{(d+1) \, \Gamma\left(\frac{d+1}{2}\right)} \,_2 F_1 \left(-\frac{d-1}{2},\frac{d+1}{2}; \frac{d+3}{2}; \sin^2\frac{A}{2}\right).
\end{align}
 Similarly, we can find a closed formula for $I_d(r,R)$ in terms of the Gauss hyper-geometric function. However, the following crude estimate shows that
 \begin{align*}
I_d(r,R)
&=\frac{\pi^\frac{d-1}{2}}{\Gamma\left(\frac{d + 1}{2}\right)}\int^{\tilde{x}}_0  \left(2R x
-x^2\right)^\frac{d-1}{2}dx
\le \frac{\pi^\frac{d-1}{2} r^{d+1}}{(d+1) R\, \Gamma\left(\frac{d + 1}{2}\right)}  .
\end{align*}
Since we are mostly concerned with relatively large values of $R$ in comparison  to $r$, we drop $I_d(r,R)$ from contention to derive the desired estimate.
\end{proof}

\begin{figure}
\centering
\includegraphics[scale=1]{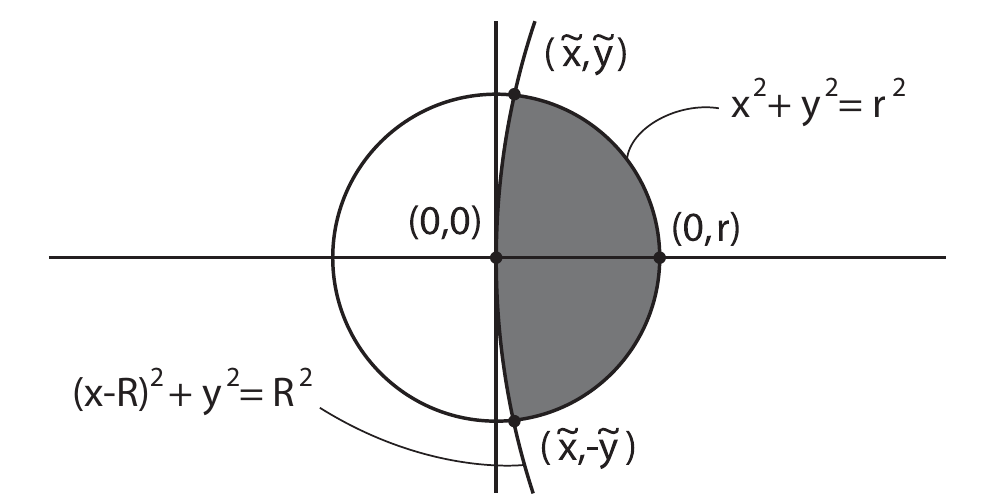}
\caption{This is the intersection of a large disk of radius $R$, and small disk of radius $r$, centered at a boundary point of the larger disk. We can assume that the center of the small disk is the origin, and we have the points of intersection of the boundaries labeled $(\tilde{x},\tilde{y})$ and $(\tilde{x},-\tilde{y})$, where
$\tilde{x}=\frac{r^2}{2R}, \tilde{y}=r\sqrt{\left(1-\frac{r^2}{4R^2}\right)}.$}
\label{2dmin}
\end{figure}

\begin{figure}
\centering
\includegraphics[scale=1]{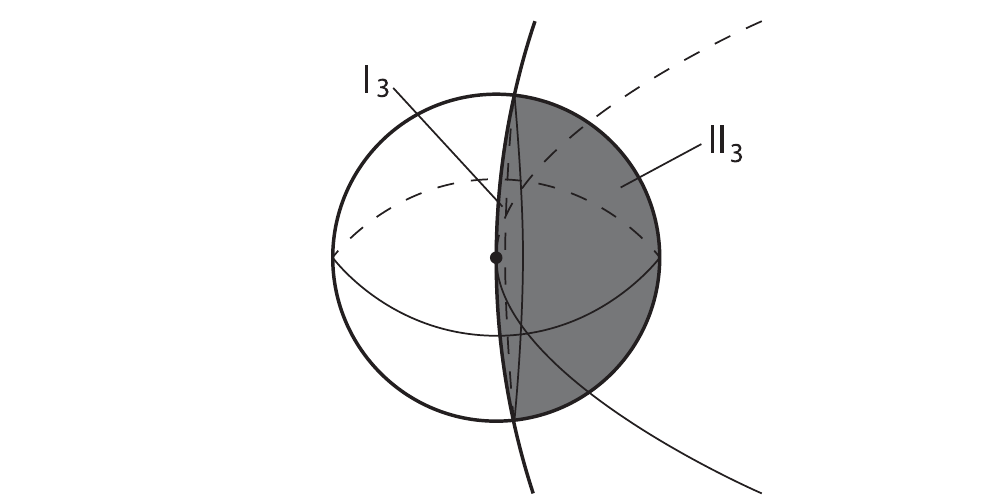}
\caption{This is the intersection of a large ball of radius $R$, and small ball of radius $r$, centered at a boundary point of the larger ball. Here, $I_3$ will measure the volume of the cap on the left, and $II_3$ will measure the volume of the cap on the right.}
\label{3dmin}
\end{figure}
%\subsection{Number of well-separated points in thin annuli }

Suppose that we are given a collection of concentric annuli with the same thickness $C n^{-1/d}$. Without loss of generality we assume that the center of these annuli is at the origin. We enumerate them from inside out: $\ca_{n,1}, \ca_{n,2}, \ldots.$ (The first one $\ca_{n,1}$ is a ball centered at $x$ and having radius $C n^{-1/d}$). For each fixed $j$, $\ca_{n,j}$ has outer radius $C\,j n^{-1/d}$. It follows that
the volume of $\ca_{n,j}$ is,
\begin{equation}\label{annMeas}
|\ca_{n,j}| = \frac{C^d \pi^\frac{d}{2}}{n \,\Gamma\left(\frac{d}{2}+1\right)}\left[ j^d - (j-1)^d \right] .
\end{equation}

\begin{proposition}\label{lemma}
Let $\cx_n$ be a discrete subset of $\R^d$ consisting of well-separated points that satisfy \eqref{well-separated}. Then the following inequality holds true:
\[
\left| \cx_n \cap \ca_{n,j}\right| \le K_d \left[ j^d - (j-1)^d \right], \quad j=1,2,\ldots,
\]
where
\[
K_d :=K_d (c,C)= 2^{\frac{3d+3}{2}}\  \left( \frac{C}{c} \right)^d.
\]
\end{proposition}

We select this format (as in the above estimate) to facilitate the telescoping technique as in the proof of Theorem 3.3.

\begin{proof}
Setting $$r=\frac{\epsilon}{2}=\frac{cn^{-\frac{1}{d}}}{2}, \quad \mbox{and} \quad R= Cj n^{-1/d}, $$
and using Equations \eqref{PHole}, \eqref{annMeas}, and the result of Proposition \ref{dProp},  we have
\begin{align} \label{final-estimate}
&\left| E_n \cap \ca_{n,j}\right| \leq \frac{|\ca_{n,j}|}{|B_\ca^d(\epsilon/2)|}  \nonumber  \\
\leq &\frac{(d+1)\ \sqrt{\pi}}{2 \sin^{(d+1)}\frac{A_j}{2} } \,  \left( \frac{C}{c} \right)^d \frac{\Gamma\left(\frac{d+1}{2}\right)}{\Gamma\left(\frac{d}{2}+1\right)} \left[\,_2 F_1 \left(-\frac{d-1}{2},\frac{d+1}{2}; \frac{d+3}{2}; \sin^2\frac{A_j}{2}\right)\right]^{-1}\left[ j^d - (j-1)^d \right].
\end{align}
Here we have
\[
\sin^2\frac{A_j}{2}:= \frac{4Cj-c}{8Cj},
\]
which satisfy the inequalities
\begin{equation}\label{ine-1}
\frac14<  \frac{4Cj-c}{8Cj} < \frac12, \quad j=1,2,\ldots.
\end{equation}
To develop a more accessible upper bound for the constant on the right hand side of Inequality \eqref{final-estimate}, we use the inequality on the left hand side of Inequality \eqref{ine-1} to write
\begin{equation}\label{3.17}
\sin^{d+1}(A/2) \geq 4^{-\frac{d+1}{2}}=2^{-(d+1)}.
\end{equation}
Furthermore, we denote
 \[
 F_{2,1}:= {\displaystyle  \,_2 F_1 \left(-\frac{d-1}{2},\frac{d+1}{2}; \frac{d+3}{2}; \sin^2\frac{A_j}{2}\right)},
 \]
 and use Equation \eqref{euler} to write
\begin{align} \label{3.18}
  F_{2,1} & \geq \frac{d+1}{2} \int_0^1t^\frac{d-1}{2} (1-t\sin^2(A/2))^\frac{d-1}{2}~dt \nonumber \\
          & > \frac{d+1}{2} \int_0^1 t^\frac{d-1}{2}\left(1-t\left(\frac{1}{2}\right)\right)^\frac{d-1}{2} ~dt \nonumber \\
          & = (d+1)\ 2^\frac{d-1}{2} \int_0^{1/2} t^\frac{d-1}{2}(1-t)^\frac{d-1}{2} ~dt \nonumber \\
          & = (d+1)\ 2^\frac{d-3}{2} \int_0^1 t^\frac{d-1}{2}(1-t)^\frac{d-1}{2} ~dt \nonumber \\
          &= (d+1)\ 2^\frac{d-3}{2} \frac{\Gamma^2 \left( \frac{d+1}{2}\right)}{\Gamma( d+1)}.
\end{align}
Here we have used the following two formulas on Beta functions:
\begin{align*}
  B(x,y) & = \int_0^1 t^{x-1} (1-t)^{y-1} dt; \quad B(x,y)= \frac{\Gamma(x) \Gamma(y)}{\Gamma(x+y)}, \quad x, y >0.
  \end{align*}
Substituting Inequalities \eqref{3.17} and \eqref{3.18} into the right hand side of Inequality \eqref{final-estimate}, and using the following  formula
\[
\Gamma(z) \Gamma(z + \frac12)= 2^{1-2z}\ \sqrt{\pi}\  \Gamma(2z)
\]
to reduce the expression involving the Gamma-function values, we obtain
\begin{align} \label{final-esti}
&\left| E_n \cap \ca_{n,j}\right| \leq 2^{\frac{3d+3}{2}}\  \left( \frac{C}{c} \right)^d \left[ j^d - (j-1)^d \right].
\end{align}
This is the desired estimate.
\end{proof}

In the remainder of this section, we discuss the special case in which an annulus has thickness $\delta$ and outer radius $(j +1)\ \delta$. This case was studied by Narcowich and Ward \cite{NW} as early as in 1991. Let $\cx$ be a discrete subset of $\R^d$ with separation radius $\delta$.  That is
 \[
  \delta = \frac12 \inf_{\substack{x \ne y\\x,y \in \cx}} |x - y|.
\]
We have the following tighter estimate, which shows that for sufficiently large $j$, the exponential growth (with respect to the dimension $d$) of the constant
can be mitigated.

\begin{corollary}\label{NW}
Let $\cx$ be a discrete subset of $\R^d$ with separation radius $\delta$. Let $\ca_{\delta,j}$ be an annulus with thickness $\delta$ and inner radius $j\delta$. Then the following inequality holds true:
\[
|\cx \cap \ca_{\delta,j}| \le 2d\ j^{d-1}\  e^{\frac{5d-3}{4j-1}}.
\]
\end{corollary}

\begin{proof}
The proof is similar to that of Proposition \ref{lemma}. In the process, however, some parameters must be adjusted, and pertinent inequalities tightened. Using the same argument leading to Inequality \eqref{PHole}, we have
\begin{equation}\label{PHole-1}
\left| \cx \cap \ca_{\delta,j} \right| \leq \frac{|\ca_{\delta,j} |}{|B_{\ca_{\delta,j}}^d(\delta)|},
\end{equation}
where $B_{\ca_{\delta,j}}^d(\delta)  $ is the minimal intersection of any ball of radius $\delta$ centered at a point in the annulus $\ca_{\delta,j}$.
Gong through a similar process as in the proof of Proposition \ref{dProp}, which entails replacing $r$ by $\delta$ and  $R$ by $(j +1)\ \delta$, and consequently $\sin^2 \frac{A}{2}$ in Proposition \ref{dProp} being $D_j:=\left(\frac12 - \frac{1}{4j}\right)$ here,
we have
\begin{align}\label{PHole-2}
 |B_{\ca_{\delta,j}}^d(\delta)| \ge & D_j^{\frac{d+1}{2}} \,
 \frac{2^{d+1} \,  \pi^\frac{d-1}{2} \, \delta^d}{(d+1) \, \Gamma\left(\frac{d+1}{2}\right)} \,_2 F_1 \left(-\frac{d-1}{2},\frac{d+1}{2}; \frac{d+3}{2}; D_j \right).
\end{align}
It follows from Inequalities \eqref{PHole-1} and \eqref{PHole-2} that
\begin{align} \label{NW-estimate}
\left| \cx \cap \ca_{\delta,j}\right|
\leq &\frac{(d+1)\ \sqrt{\pi}}{2^{d+1}\ D_j^{\frac{d+1}{2}}}
 \frac{\Gamma\left(\frac{d+1}{2}\right)}{\Gamma\left(\frac{d}{2}+1\right)} \left[\,_2 F_1 \left(-\frac{d-1}{2},\frac{d+1}{2}; \frac{d+3}{2}; D_j \right)\right]^{-1}
 \left[  (j+1)^d - j^d \right].
\end{align}
As in the last part of the proof of Proposition \ref{lemma}, we show that
\begin{align} \label{NW-estimate-1}
 &(d+1)\ \sqrt{\pi}
 \frac{\Gamma\left(\frac{d+1}{2}\right)}{\Gamma\left(\frac{d}{2}+1\right)} \left[\,_2 F_1 \left(-\frac{d-1}{2},\frac{d+1}{2}; \frac{d+3}{2}; D_j \right)\right]^{-1}
 \le 2^{ \frac{d+3}{2}}.
\end{align}
We estimate the remaining pieces involving $j$ on the right hand of Inequality \eqref{NW-estimate} as follows.
\begin{align} \label{NW-estimate-2}
D_j^{-\frac{d+1}{2}}=&\left(\frac12 - \frac{1}{4j}\right)^{-\frac{d+1}{2}}
\le  2^{\frac{d+1}{2}}\ e^{\frac{d+1}{2(2j-1)}}.
\end{align}
\begin{align} \label{NW-estimate-3}
&(j+1)^d - j^d \le d\ (j+1)^{d -1} \le d\ j^{d -1}\ e^{\frac{d-1}{j}}.
\end{align}
Substituting inequalities \eqref{NW-estimate-1}, \eqref{NW-estimate-2}, and \eqref{NW-estimate-3} into the right hand side of Inequality \eqref{NW-estimate}, we get the desired estimate.
\end{proof}

\section{Optimal order Jackson type error estimate}

 We are interested in the class ${\mathbb K}_{\kappa, \alpha}$ of functions $K$ that satisfy the following conditions:
 \begin{description}
  \item{(i)}  $K: \R^{d} \rightarrow \R_+ $ is continuous on $\R^{d}$.
  \item{(ii)}  ${\dis m_1:= \min_{|x| \le 1} K(x) >0 }$.
  \item{(iii)} For a fixed $\alpha >0$, there holds $ K(x) \le \kappa (1+|x|^2)^{-\alpha}, \quad x \in \R^{d}$, \quad where $\kappa$ is a constant independent of $x$.
%\item{(iv)} $\int_{\R^{d+1}} K(x) dx=1$.
\end{description}
When the scaling factor $\frac1\epsilon$ is introduced, the decay condition in (iii) can be equivalently written as the following:
\begin{equation} \label{decay-condition}
K(\epsilon^{-1} x) \le \kappa  \min \left( 1, (|x|/\epsilon)^{-2\alpha}\right), \quad 0<\epsilon \le 1, \quad x \in \R^{d}.
\end{equation}
This is the form we will be using the most often throughout this paper. Suppose that values of a function $f \in C(X)$ are available at every point of a discrete set $\cx_n \subset X$. Assume that $\cx_n$ is quasi-uniformly distributed, and satisfy Inequalities \eqref{well-separated} and \eqref{fill-dist}.
We study the operator $\ct_n$ on $C(X) $:
\[
\ct_n:  f \mapsto \sum_{y \in \cx_n}f(y)K_n(x -y),
\]
where $K_n(x)=K\left(\beta_n x\right), \; \beta_n:=C^{-1}\ n^{1/d}$, and $K \in {\mathbb K}_{\kappa, \alpha}$.
%We therefore call such an operator ``quasi-Monte Carlo operator".
%The approximation power of quasi-Monte Carlo operators thus defined depends on the scalability and the decay rate (off the diagonal of $X \times X$) of the kernel $K$ and the configurations of the point set $\cx_n$, and to a lesser extent, on the differential and topological structure of the domain $X$.
We employ  the ``rational formation":
\[
\frac{1}{S_{K,n}(x)} \sum_{y \in \cx_n}f(y)K_n(x - y),\quad \mbox{where} \quad S_{K,n}(x)=\sum_{y \in \cx_n} K_n(x - y),
\]
to approximate a bounded and continuous function $f$ on $X$, and derive an optimal order Jackson type error estimate (Theorem \ref{gen-th}) with an explicit constant. (We refer readers to the historically-important result of Jackson \cite{jackson}.) The construction of approximants here is reminiscent of the original Shepard approximation. The new contribution here is the incorporation of the scaling factor $\beta_n$, which is crucial for the optimality of the approximation order.
The feasibility of this approximation scheme depends on the condition:  $S_{K,n}(x)\ne 0, \; x \in X$, which we will address in Lemma \ref{gen-le}. It is worth noting that the above rational formations of shifts of the basis function $K_n$ give rise to a linear operator on $C(X)$. In the literature, these are called ``quasi-interpolation operators."
%If a fixed kernel $K$ is employed, then the rational formations of the approximators we employ here resemble those in ``Shepard approximation" \cite{shepard}.

% In Part (iii) above, we have used the symbol ``$\ll$" which is widely credited to be one of Vinogradov's inventions. Precisely, the above inequality means that there exists a constant $C$ independent of $x$, such that
%\begin{equation*}
%K(x) \le C (1+|x|^2)^{-\alpha}.
%\end{equation*}

Let $X \subset \RR^{d}$ be a convex domain. Let $BC(X)$ denote the normed linear space consisting of bounded and continuous functions on $X$. Here
 the norm of $f \in BC(X)$ is defined by
\[
\|f\|_{BC}:= \inf \{M: |f(x)| \le M, \; \mbox{for all} \; x \in X \}.
\]
%Let $K_n(x) = K(\beta_n x),$ where $\beta_n=C n^{1/d}$. Define the operator sequence $\ct_n$:
%\[
%\ct_n: BC(X) \ni f  \mapsto \ct_n(f), \quad \ct_n(f)(x)= \sum_{y \in \cx_n} f(y) K_n ( x -y).
%\]
%Operators given in this format are analogous to the  classical quasi-Monte Carlo method in evaluating integrals. Therefore, we call them quasi-Monte Carlo operators.
\newcommand{\bk}{{\mathbb K}_{\kappa, \alpha}}
 %++++++++++++++++++imported
 \begin{proposition} \label{gen-pro}
Assume that the points in $\cx_n$ are quasi-uniformly distributed satisfying Inequalities \eqref{well-separated} and \eqref{fill-dist}, and  that $K \in \bk$ for $\alpha > \frac{d+1}{2}$. Then the operator sequence $\ct_n$ is uniformly bounded on $BC(X)$.
More precisely, we have
\[
\|\ct_n \| \le  \kappa \, K_d \, C_{\alpha,d}, \quad n \in \mathbb{N},
\]
where ${\dis C_{\alpha,d}:=  1  + \sum_{k=1}^d {d \choose k}\ \frac{2\alpha-k + 2}{2\alpha-k + 1}, }$ and $K_d:=K_d(c,C)$ is a constant depending only on $d, c,$ and $C$. The explicit expression of $K_d$  is given in Proposition \ref{lemma}.
\end{proposition}
%\end{document}
\begin{proof}
It suffices to show that
\[
\sup_{x \in X}\sum_{y \in \cx_n}  K_n ( x -y) \le \kappa \, K_d\  C_{\alpha,d}.
\]
Fix an $x \in X$.  For $j \in \N$, let $\ca_{n,j}=  \{y: (j-1) \beta_n^{-1} \le |y-x| < j \beta_n^{-1}\}$,
 which are concentric annuli centered at $x$ and having thickness $\beta_n^{-1}$. Denote $Z_{n,j}=X \bigcap \ca_{n,j}$.   For each fixed $n$, we have
\[
\R^{d}= \bigcup^\infty_{j=1} \ca_{n,j},
\]
which induces a decomposition of  $\cx_n$,
\[
\cx_n= \bigcup^\infty_{j=1} Z_{n,j}, \quad Z_{n,j}:=\cx_n \cap \ca_{n,j}.
\]
We use the decay condition in Inequality \eqref{decay-condition} to write
\[
K(\beta_n(x-y)) \le \frac{\kappa}{(j-1)^{2 \alpha}}, \quad y \in Z_{n,j}, \quad j \ge 2.
\]
Making use of the upper bound estimate for  $|Z_{n,j}|$ ( Proposition \ref{lemma}), we have
\begin{align}\label{esti}
 \sum_{y \in \cx_n}  K_n ( x -y) \le &  \sum_{j=1}^\infty \sum_{y \in Z_{n,j}}  K (\beta_n ( x -y)) \nonumber  \\
  \le &  \kappa\ K_d \left(|Z_{n,1}| + \sum_{j=2}^\infty |Z_{n,j}| (j-1)^{-2\alpha}\right) \nonumber  \\
  \le & \kappa\ K_d \left( 1  + \sum_{j=1}^\infty \left[ (j+1)^d - j^d \right] j^{-2\alpha}\right).
\end{align}
It is easy to see that $C_{\alpha, d}$ is an upper bound for the expression inside the parenthesis above. Since $x \in X$ is arbitrarily chosen, the proof is complete.
\end{proof}

% \nonumber  \\
%  \le & \kappa\ K_d \left( 1  + \sum_{k=1}^d {d \choose k} \sum_{j=1}^\infty \left[ (j+1)^d - j^d \right] j^{-2\alpha}\right) \nonumber  \\
%\le & \kappa\ K_d \left( 1  + \sum_{k=1}^d {d \choose k} \frac{2\alpha-k + 2}{2\alpha-k + 1}\right).
%
%Let $S_{K,n}$ denote the function
%\[
%S_{K,n}(x): = \sum_{y \in \cx_n}  K_n ( x -y)=\sum_{y \in \cx_n}  K (\beta_n( x -y)), \quad x \in X.
%\]

\begin{lemma} \label{gen-le}
Let $X$ be a convex domain in $\R^d$. Suppose $\cx_n$ is quasi-uniformly distributed in $X$ satisfying Inequalities \eqref{well-separated} and \eqref{fill-dist}, and  that $K \in \bk$. Then we have
\[
 S_{K,n}(x) \ge  m_1, \quad x \in X.
\]
\end{lemma}

\begin{proof}
Recall that
${\displaystyle m_1= \min_{0 \le |x|  \le 1} K(x)}$.
 Since $h_{\cx_n}$, the Hausdorff distance between the two sets $\cx_n$ and $X$, is less than or equal to $C n^{-1/d}$, for any fixed $x \in X$, there is at least one $x_{j_0}\; ( 1 \le j_0 \le n)$ such that $|x_{j_0}- x| \le C n^{-1/d}$. It follows that
\[
S_{K,n}(x) = \sum_{y \in \cx_n} K(\beta_n (x-y))  \ge m_1.
\]
The proof is complete.
\end{proof}
\newcommand{\cf}{\mathcal F}

Let $f \in BC(X)$. We consider approximating $f$ by the function $\cf_n$:
\[
\cf_n(x)=[S_{K,n}(x)]^{-1} \sum_{y \in \cx_n}f(y) K(\beta_n(x-y)).
\]
The function $\cf_n$ is constructed using a rational formation of shifts of the base function $K_n$.

To gauge the order of approximation, we make use of the modulus of continuity $\omega(f,t)$ of $f$ which is
a function from
$\R_+$ to $\R_+$, defined by:
\[
\omega(f,t)= \sup_{\substack{x,y \in X \\|x-y|\le t}} |f(x)-f(y)|.
\]
The function $\omega(f,t)$ satisfies the following property often referred to as ``sub-additivity":
\begin{equation} \label{subadditivity}
\omega(f,\gamma\ t) \le (\gamma +1) \omega(f,t), \quad t \ge 0, \quad \gamma \ge 0.
\end{equation}

We have the following Jackson-type approximation error estimate.
\begin{theorem}\label{gen-th}
Let $X$ be a convex subset of $\R^d$. Suppose that $\cx_n$ is quasi-uniformly distributed in $X$ satisfying Inequalities \eqref{well-separated} and \eqref{fill-dist}, and  that $K \in \bk$ for $\alpha > \frac{d+2}{2}$. Then, for every $f \in BC(X)$, we have
 \[
\|[S_{K,n}(x)]^{-1}\ct_n(f)(x) - f(x)\|_{BC}  \le  \kappa\ m^{-1}_1\ (C+1)\ K_d \ C^*_{\alpha, d}\ \omega\,(f, n^{-1/d}),
\]
in which
\[
 C^*_{\alpha, d}=\sum^d_{k=0}{d+1 \choose k} \frac{2 \alpha -k +2}{2 \alpha -k +1} - \frac{1}{2 \alpha -d +1}.
 \]
\end{theorem}

\begin{proof}
 For a given $f \in BC(X)$, we  write
\begin{equation}\label{esti3}
f(x) = [S_{K,n}(x)]^{-1} \sum_{y \in \cx_n} f(x) K(\beta_n (x-y)), \quad  x \in X.
\end{equation}
Fix an $ x \in X$. Let $Z_{n,j}$ be as defined in the proof of Proposition \ref{gen-pro}. Fix a $y \in Z_{n,j}$, join $y$ and $x$ by a straight line segment, on which we take $(j-1)$ points $t_\nu \,(\nu =1,\ldots, (j-1))$, and let $t_0 := x$ and $t_{k} := y$ so that $|t_{\nu-1} - t_\nu| \le C\ n^{-1/d}\,(\nu =1,\ldots, j)$. By the sub-additivity property of $\omega(f,t)$, we have
\begin{equation} \label{estimate1}
|f(y) - f(x)| \le \sum^{j}_{\nu=1} |f(t_\nu) - f(t_{\nu-1})| \le j \ \omega(f,C\ n^{-1/d} ) \le (C+1) j \ \omega(f,n^{-1/d} ).
\end{equation}
By Proposition \ref{lemma}, Inequalities  \eqref{decay-condition}, \eqref{estimate1},   we have
\begin{align*}
     &|f(x)-[S_{K,n}(x)]^{-1}\ct_n(f)(x)| \\
 \le & m_1^{-1} \sum_{y \in \cx_n} \left|f(x) - f(y)\right| K(\beta_n (x-y)) \\
 \le & m_1^{-1}   \sum^\infty_{j=1} \sum_{y \in Z_{n,j}} \left|f(x) - f(y)\right| K(\beta_n (x-y))   \\
 %\le & \kappa\ m_C^{-1} (C+1) K_d(c,C)\ \omega(f,n^{-1/d} )\sum^\infty_{j=1}  \left[j^{d+1}- j(j-1)^{d}\right] j^{-2\alpha} \\
   \le               &  \kappa\ m_1^{-1} (C+1) K_d(c,C)\ \omega(f,n^{-1/d} )\left\{ 1+ \sum^\infty_{j=1}  \left[(j+1)^{d+1}- (j+1)\ j^{d}\right] j^{-2\alpha} \right\}.
\end{align*}
It is easy to see that $C^*_{\alpha, d}$ is an upper bound for the expression inside the curly braces above.
Since the above estimate is true for an arbitrary $x \in X$, the desired result follows.
\end{proof}
%
%The sum inside the curly braces above can be estimated as follows.
%\begin{align*}
% & \sum^\infty_{j=1}  \left[(j+1)^{d+1}- (j+1)\ j^{d} \right] j^{-2\alpha}  \\
%=& \sum^\infty_{j=1}  \left[\sum^{d+1}_{k=0}{d+1 \choose k} j^k - j^{d+1}-j^d \right] j^{-2\alpha} \\
%=&  \sum^d_{k=0}{d+1 \choose k}\sum^\infty_{j=1} j^{-2\alpha+k} -\sum^\infty_{j=1}j^{-2\alpha +d}   \\
%<&  \sum^d_{k=0}{d+1 \choose k}\left[1 + \int^\infty_1 t^{-2\alpha+k} dt \right]
% - \int^\infty_1 t{-2\alpha +d} dt  \\
%=&  \sum^d_{k=0}{d+1 \choose k} \frac{2 \alpha -k +2}{2 \alpha -k +1} - \frac{1}{2 \alpha -d +1}.
%\end{align*}

%\noindent{\bf Achnowledgement:} We are grateful to an anonymous referee who has made valuable suggestions and corrections to a previous version of the paper.

\noindent Department of Mathematics\\
\noindent Missouri State University \\
\noindent Springfield, MO 65897 \\
\noindent USA

\bigskip

\noindent Department of Mathematics\\
\noindent Missouri State University \\
\noindent Springfield, MO 65897 \\
\noindent USA

\bigskip

\noindent Shanghai Key Laboratory for Contemporary Applied Mathematics\\
\noindent School of Mathematical Science \\
\noindent Fudan University, Shanghai\\
\noindent China

\end{document}